\documentclass[11pt]{article}
\usepackage[a4paper,margin=1in]{geometry}
\usepackage[T1]{fontenc}
\usepackage[utf8]{inputenc}
\usepackage{lmodern}
\usepackage{microtype}
\usepackage{amsmath,amssymb,amsthm,mathtools}
\usepackage{hyperref}
\usepackage{enumitem}
\usepackage{mathrsfs}

\hypersetup{colorlinks=true,linkcolor=blue,citecolor=blue,urlcolor=blue}

\newtheorem{theorem}{Theorem}[section]
\newtheorem{proposition}[theorem]{Proposition}

\newtheorem{remark}[theorem]{Remark}
\newtheorem{definition}[theorem]{Definition}

\newcommand{\Q}{\mathbf{Q}}
\newcommand{\Pj}{\mathbf{P}}
\newcommand{\Eul}{\mathcal{E}}

\newcommand{\keywords}[1]{\par\medskip\noindent\textbf{Keywords: }#1}

\title{Quartic Rational Diophantine Quadruples and the Euler Surface}
\author{%
Alen Andra\v{s}ek\\
{\small University of Zagreb, Faculty of Civil Engineering, Department of Mathematics}
\and
Matija Kazalicki\\
{\small University of Zagreb, Faculty of Science, Department of Mathematics}
\and
Domagoj Vlah\\
{\small University of Zagreb Faculty of Electrical Engineering and Computing, Department of Applied Mathematics}%
}
\date{}

\begin{document}
\maketitle

\begin{abstract}
We prove that there exist infinitely many quartic rational Diophantine quadruples, that is, sets of four pairwise distinct nonzero rational numbers whose pairwise products increased by \(1\) are fourth powers in \(\Q\). To the best of our knowledge, no examples of such quadruples were previously known. Our construction is motivated by computer experiments and leads naturally to the classical Euler surface
\[
\Eul:\qquad X^4+Y^4=Z^4+W^4.
\]
We show that every rational point on a suitable Zariski-open subset of \(\Eul\) yields a quartic rational Diophantine quadruple, thereby obtaining a rational map from the Euler surface to the parameter space of quartic quadruples. In particular, Euler's classical parametrization produces the first explicit infinite family of quartic rational Diophantine quadruples. We also explain that the same mechanism extends to arbitrary exponents \(k\geq 2\), with the Euler surface replaced by the Fermat--Euler surface \(\mathcal E_k:X^k+Y^k=Z^k+W^k\).

For even \(k\), every rational point on a suitable open subset of \(\mathcal E_k\) gives rise to a \(k\)th power rational Diophantine quadruple, while for odd \(k\) one obtains such quadruples on the locus where \(W/Z\) is a square.
\end{abstract}
\keywords{higher power rational Diophantine tuples, quartic Diophantine quadruples, Euler surface, Fermat-Euler surface}
\section{Introduction}

\begin{definition}
Let $k\geq 2$ and $m\geq 2$ be integers. A \emph{$k$th power rational Diophantine $m$-tuple} is a set of pairwise distinct nonzero rational numbers
\[
\{a_1,a_2,\dots,a_m\}\subset \Q
\]
with the property that for every $1\leq i<j\leq m$ the number $a_i a_j+1$ is a $k$th power in $\Q$.
\end{definition}

The case $k=2$ is the classical one: such sets are usually called \emph{rational Diophantine $m$-tuples}. They go back to Diophantus, who found the rational quadruple
\[
\left\{\frac1{16},\frac{33}{16},\frac{17}{4},\frac{105}{16}\right\},
\]
while Fermat found the integer quadruple $\{1,3,8,120\}$. In the rational setting the theory is especially rich. Euler proved that there exist infinitely many rational Diophantine quintuples, Gibbs found the first rational Diophantine sextuple, and Dujella, Kazalicki, Miki\'c and Szikszai later proved that there are infinitely many rational Diophantine sextuples \cite{DKMS}. Subsequent work produced several further infinite families of rational Diophantine sextuples, including families with additional structure \cite{DK,DKP-sq,DKP-strong}. No examples of rational Diophantine septuples are known. For a broader historical overview and further references, see Dujella's book \emph{Diophantine $m$-tuples and elliptic curves} \cite{DujellaBook}.

The higher-power variant, in which one fixes an exponent \(k>2\) and asks
that \(a_i a_j+1\) be a \(k\)-th power, was studied in the integer setting by
Bugeaud and Dujella \cite{BD}. They exhibited cubic and quartic examples of
integer triples and proved, among other results, that there are no \(k\)-th
power Diophantine quadruples in positive integers for \(k\ge 177\). Recently,
Batta, Szikszai and Tengely initiated the systematic study of higher-power
rational Diophantine tuples. They proved that, for every \(k\ge 2\), there
exist infinitely many \(k\)-th power rational Diophantine triples; in the
cubic case, they also proved the existence of infinitely many rational cubic
Diophantine quadruples \cite{BST}. In a different direction, Byeon and Fuchs
showed that every rational cube Diophantine pair can be extended to a rational
cube Diophantine triple \cite{BF}.

We also mention the recent preprint of Andra\v{s}ek \cite{Andrasek}, which
introduces regular quartic Diophantine triples, reduces their construction to
rational points on an associated elliptic surface, and obtains several infinite
parametric families, including infinitely many triples with all elements
positive. By contrast, to the best of our knowledge, no example of
a rational quartic Diophantine quadruple had previously appeared in the
literature.

The goal of the present paper is to treat the quartic case in detail and prove that quartic rational Diophantine quadruples exist in infinite families. We then show that the same construction extends naturally to arbitrary exponents.

\begin{theorem}\label{thm:main-intro}
There exist infinitely many quartic rational Diophantine quadruples.
\end{theorem}

Our construction is governed by the Euler surface
\[
\Eul:\qquad X^4+Y^4=Z^4+W^4.
\]
The point is that a rational point on $\Eul$ canonically produces a quartic Diophantine quadruple of a special type.

\begin{proposition}\label{prop:euler-to-quadruple}
Let $(X:Y:Z:W)\in \Eul(\Q)$ with $ZW\neq 0$, and define
\begin{equation}\label{eq:euler-map}
 a=\frac{X^4-W^4}{Z^2W^2},\qquad
 b=-\frac{W^2}{Z^2},\qquad
 c=\frac{Y^4-W^4}{Z^2W^2},\qquad
 d=\frac{Z^2}{W^2}.
\end{equation}
Then
\[
ab+1=\left(\frac{Y}{Z}\right)^4,
\qquad
ac+1=\left(\frac{XY}{ZW}\right)^4,
\qquad
ad+1=\left(\frac{X}{W}\right)^4,
\]
\[
bc+1=\left(\frac{X}{Z}\right)^4,
\qquad
bd+1=0,
\qquad
cd+1=\left(\frac{Y}{W}\right)^4.
\]
Moreover, the four numbers in \eqref{eq:euler-map} are pairwise distinct and nonzero if and only if
\begin{equation}\label{eq:nondeg-euler}
XYZW\,(X^2-Y^2)(X^2-W^2)(Y^2-W^2)\neq 0.
\end{equation}
In that case they form a quartic rational Diophantine quadruple.
\end{proposition}

Thus the quartic problem is linked to the arithmetic of the Euler surface. Euler's classical construction gives a rational curve on the Euler surface $\Eul$:
\begin{align*}
X&=\alpha^7+\alpha^5\beta^2-2\alpha^3\beta^4+3\alpha^2\beta^5+\alpha\beta^6,\\
Y&=\alpha^6\beta-3\alpha^5\beta^2-2\alpha^4\beta^3+\alpha^2\beta^5+\beta^7,\\
Z&=\alpha^7+\alpha^5\beta^2-2\alpha^3\beta^4-3\alpha^2\beta^5+\alpha\beta^6,\\
W&=\alpha^6\beta+3\alpha^5\beta^2-2\alpha^4\beta^3+\alpha^2\beta^5+\beta^7.
\end{align*}
and for this parametrization one has $X^4+Y^4=Z^4+W^4$. Composing this parametrization with Proposition~\ref{prop:euler-to-quadruple} yields a family of quartic rational Diophantine quadruples.

If we specialize Euler's parametrization by $(\alpha,\beta)=(1,t)$ then we obtain the rational point
\[
(X:Y:Z:W)=\bigl(
t^6+3t^5-2t^4+t^2+1:
t(t^6+t^4-2t^2-3t+1):
t^6-3t^5-2t^4+t^2+1:
t(t^6+t^4-2t^2+3t+1)
\bigr)
\]
on \(\Eul\). Hence, for all \(t\in\Q\) outside a finite exceptional set,
Proposition~\ref{prop:euler-to-quadruple} gives the quartic rational
Diophantine quadruple
\[
\begin{aligned}
a(t)&=
-\frac{(t-1)(t+1)(t^2+1)(t^4+1)(t^2-t-1)(t^2+t-1)(t^4-t^2+1)}%
{t^2\,(t^6+t^4-2t^2+3t+1)^2\,(t^6-3t^5-2t^4+t^2+1)^2}\\
&\qquad\times
(t^6+4t^4-6t^3+4t^2+1)(t^6+4t^4+6t^3+4t^2+1),\\[1ex]
b(t)&=
-\frac{t^2\,(t^6+t^4-2t^2+3t+1)^2}{(t^6-3t^5-2t^4+t^2+1)^2},\\[1ex]
c(t)&=
-\frac{24t^3(t^2+1)(t^4+3t^2+1)(t^6-2t^4+t^2+1)(t^6+t^4-2t^2+1)}%
{(t^6+t^4-2t^2+3t+1)^2\,(t^6-3t^5-2t^4+t^2+1)^2},\\[1ex]
d(t)&=
\frac{(t^6-3t^5-2t^4+t^2+1)^2}{t^2\,(t^6+t^4-2t^2+3t+1)^2}.
\end{aligned}
\]

Theorem~\ref{thm:main-intro} follows immediately from Proposition~\ref{prop:euler-to-quadruple}, since Euler's parametrization supplies infinitely many rational points on $\Eul$ satisfying the nondegeneracy condition \eqref{eq:nondeg-euler}. In particular, our construction places quartic rational Diophantine quadruples inside the geometry of a classical K3 surface.

Although the main focus of the paper is the quartic case, the construction we obtain is not peculiar to exponent $4$. Once one isolates the special locus $v=0$ and $ru=tw$, the quartic compatibility conditions are governed by the Euler surface, and the same argument with exponent $4$ replaced by an arbitrary integer $k\geq 2$ leads to the Fermat--Euler surface
\[
\mathcal E_k:\qquad X^k+Y^k=Z^k+W^k.
\]
In Theorem \ref{thm:general-euler-to-quadruple} we formulate this general construction explicitly. For even $k$, rational points on a suitable open subset of $\mathcal E_k$ produce $k$th power rational Diophantine quadruples, while for odd $k$ the same holds under the additional condition that $W/Z$ be a square. In particular, the cubic family of Batta--Szikszai--Tengely appears as a special case of our construction.

\section{Experiments and the Euler surface}\label{sec:euler-surface}

Let $\{a,b,c,d\}$ be a quartic rational Diophantine quadruple. After choosing fourth roots, we may write
\begin{equation}\label{eq:rstuvw-def}
\begin{alignedat}{2}
ab+1&=r^4,  &\qquad ad+1&=u^4, \\
ac+1&=s^4,  &\qquad bd+1&=v^4, \\
bc+1&=t^4,  &\qquad cd+1&=w^4.
\end{alignedat}
\end{equation}

Then
\begin{equation}\label{eq:abcd-from-rstuvw}
a^2=\frac{(r^4-1)(u^4-1)}{v^4-1},\qquad
b=\frac{r^4-1}{a},\qquad
c=\frac{s^4-1}{a},\qquad
d=\frac{u^4-1}{a},
\end{equation}
and the remaining pairwise conditions are equivalent to
\begin{equation}\label{eq:compatibility-second}
(r^4-1)(w^4-1)=(t^4-1)(u^4-1),
\end{equation}
\begin{equation}\label{eq:compatibility-third}
(r^4-1)(w^4-1)=(s^4-1)(v^4-1).
\end{equation}
Conversely, any $r,s,t,u,v,w\in\Q$ satisfying \eqref{eq:compatibility-second}, \eqref{eq:compatibility-third}, together with the square condition in \eqref{eq:abcd-from-rstuvw}, determine a quartic rational Diophantine quadruple through \eqref{eq:abcd-from-rstuvw}.

These identities guided our first computations. Starting from a pair $(a,b)$ with
\[
b=\frac{r^4-1}{a},
\]
we searched for rational numbers $x$ such that $ax+1$ is a fourth power and $bx+1$ is a square, hoping that in favourable cases this square would itself be a fourth power for two values $x=c,d$.

More precisely, the pair $(a,b)$ was not sampled uniformly from all rational pairs. Fix a height bound $B$. The search first enumerates positive reduced rationals $a=p/q$ deterministically by increasing height
\[
h(a)=\max\{|p|,q\},
\]
namely $1$ at height $1$, and for each $h\ge 2$ the fractions
\[
\frac1h,\frac2h,\ldots,\frac{h-1}{h},\frac{h}{h-1},\ldots,\frac{h}{1},
\]
with reducible fractions omitted. For each such $a=p/q$, it then enumerates coprime positive integers $u,v$ with $u\neq v$ satisfying
\[
|p|v^4\le B,\qquad |u^4-v^4|\le \left\lfloor \frac{B}{q}\right\rfloor,
\]
and sets
\[
b=\frac{q(u^4-v^4)}{pv^4}=\frac{(u/v)^4-1}{a}.
\]
Thus $ab+1=(u/v)^4$ is a fourth power by construction. A generated pair is retained only if $b\neq 0$, $a\neq b$, and
\[
\max\{h(a),h(b)\}\le B;
\]
moreover, the canonical orientation
\[
h(a)<h(b)\qquad\text{or}\qquad \bigl(h(a)=h(b)\ \text{and}\ a\le b\bigr)
\]
is imposed so that the two orderings of the same pair are not counted separately. For each retained pair, the subsequent search for $x$ is reduced to a brute-force search for rational points on
\[
(av)^2=ab\,u^4+a(a-b),
\]
which is equivalent to the system $ax+1=u^4$, $bx+1=v^2$, with $x=(u^4-1)/a$.

In the computation reported here we took $B=100\,000$, so in particular both $h(a)$ and $h(b)$ were at most $100\,000$. To parallelize the deterministic enumeration, the index range for the parameter $a$ was partitioned into contiguous intervals, each contributing at most a prescribed number of valid pairs, and these intervals were processed separately. Altogether, this search examined $219{,}415{,}142$ pairs $(a,b)$ generated in this way. It yielded $569$ unique triples, among which $11$ almost quadruples. The cumulative used CPU time was $11{,}785$ hours, corresponding to an average of about $193$ ms per pair.

Two particularly suggestive almost quadruples examples were
\[
(a,b,c,d)=\left(\frac{64}{9},-\frac{2295}{16384},\frac{145}{36},-\frac{317135}{2359296}\right),
\]
for which
\[
(r,s,t,u,w)=\left(\frac14,\frac73,\frac{13}{16},\frac{11}{24},\frac{79}{96}\right),
\]
and
\[
(a,b,c,d)=\left(384,-\frac{85}{32768},\frac{435}{2},-\frac{317135}{127401984}\right),
\]
for which
\[
(r,s,t,u,w)=\left(\frac14,17,\frac{13}{16},\frac{11}{24},\frac{79}{96}\right).
\]
In particular, these examples shared the same values of $(r,t,u,w)$ and differed only in $s$. This made it natural to fix $(r,t,u,w)$ satisfying \eqref{eq:compatibility-second} and to study the remaining conditions.

For fixed $(r,t,u,w)$ satisfying \eqref{eq:compatibility-second}, the square condition in \eqref{eq:abcd-from-rstuvw} becomes the genus one curve
\begin{equation}\label{eq:genus-one-curve}
\mathcal{C}_{r,u}:\qquad y^2=(r^4-1)(u^4-1)(x^4-1),
\end{equation}
with $x=v$. We therefore varied $(r,t,u,w)$ satisfying \eqref{eq:compatibility-second}, computed rational points on \eqref{eq:genus-one-curve}, and tested whether they also satisfy \eqref{eq:compatibility-third}. This produced several genuine quartic quadruples, including
\[
(a,b,c,d)=\left(\frac{310300575}{317623684},-\frac{17689}{17956},-\frac{75195840}{79405921},\frac{17956}{17689}\right),
\]
with
\[
(r,s,t,u,v,w)=\left(\frac{59}{134},\frac{4661}{8911},\frac{79}{67},\frac{158}{133},0,\frac{59}{133}\right),
\]
and
\[
(a,b,c,d)=\left(\frac{310300575}{86899684},-\frac{3481}{24964},\frac{75195840}{21724921},\frac{24964}{3481}\right),
\]
with
\[
(r,s,t,u,v,w)=\left(\frac{133}{158},\frac{8911}{4661},\frac{67}{79},\frac{134}{59},0,\frac{133}{59}\right).
\]
The striking feature of these examples is that they satisfy
\[
v=0,\qquad ru=tw.
\]
This suggested restricting attention to quartic quadruples with precisely this additional structure.

It is convenient to rewrite the relation $ru=tw$ in the form
\begin{equation}\label{eq:def-k}
\frac{r}{w}=\frac{t}{u}=\kappa,
\end{equation}
so that
\begin{equation}\label{eq:r-t-from-k}
r=\kappa w,\qquad t=\kappa u.
\end{equation}
We then set
\begin{equation}\label{eq:def-s-special}
s=\kappa uw.
\end{equation}
The point of the specialisation is that, away from the diagonal locus $u^4=w^4$, both compatibility conditions collapse to a single quartic relation, while the square condition becomes automatic.

\begin{proposition}\label{prop:quartic-reduction}
Let $\kappa,u,w\in\Q^\times$ and define $r,t,s$ by \eqref{eq:r-t-from-k}, \eqref{eq:def-s-special}, and $v=0$. Then the third compatibility condition \eqref{eq:compatibility-third} is equivalent to
\begin{equation}\label{eq:special-quartic}
u^4+w^4=1+\frac{1}{\kappa^4}.
\end{equation}
Moreover, the second compatibility condition \eqref{eq:compatibility-second} becomes
\begin{equation}\label{eq:diagonal-or-quartic}
(w^4-u^4)\bigl(\kappa^4(u^4+w^4)-(\kappa^4+1)\bigr)=0.
\end{equation}
In particular, away from the diagonal locus $u^4=w^4$, the second and third compatibility conditions are both equivalent to \eqref{eq:special-quartic}.
\end{proposition}

\begin{proof}
Since $v=0$, equation \eqref{eq:compatibility-third} becomes
\[
(r^4-1)(w^4-1)=1-s^4.
\]
Substituting $r=\kappa w$ and $s=\kappa uw$, we obtain
\[
(\kappa^4w^4-1)(w^4-1)=1-\kappa^4u^4w^4.
\]
Expanding and simplifying gives
\[
w^4\bigl(\kappa^4(u^4+w^4)-(\kappa^4+1)\bigr)=0.
\]
Since $w\neq 0$, this is equivalent to \eqref{eq:special-quartic}.

For the second compatibility condition we substitute $r=\kappa w$ and $t=\kappa u$ into \eqref{eq:compatibility-second} and get
\[
(\kappa^4w^4-1)(w^4-1)=(\kappa^4u^4-1)(u^4-1).
\]
After expansion the left-hand side minus the right-hand side factors as
\[
(w^4-u^4)\bigl(\kappa^4(u^4+w^4)-(\kappa^4+1)\bigr),
\]
which is exactly \eqref{eq:diagonal-or-quartic}.
\end{proof}

\begin{proposition}\label{prop:auto-square-special}
Let $\kappa,u,w\in\Q^\times$ and define $r,t,s,v$ by \eqref{eq:r-t-from-k}, \eqref{eq:def-s-special}, and $v=0$. If \eqref{eq:special-quartic} holds, then the square condition in \eqref{eq:abcd-from-rstuvw} is automatic. More precisely,
\[
(r^4-1)(u^4-1)(v^4-1)=\left(\frac{r^4-1}{\kappa^2}\right)^2.
\]
\end{proposition}

\begin{proof}
Since $v=0$, the left-hand side is
\[
-(r^4-1)(u^4-1).
\]
Using $r=\kappa w$ and \eqref{eq:special-quartic}, we compute
\[
u^4-1=\frac{1}{\kappa^4}-w^4=-\frac{\kappa^4w^4-1}{\kappa^4}=-\frac{r^4-1}{\kappa^4}.
\]
Therefore
\[
(r^4-1)(u^4-1)(v^4-1)=-(r^4-1)(u^4-1)=\frac{(r^4-1)^2}{\kappa^4},
\]
as claimed.
\end{proof}

Thus the special locus
\[
v=0,\qquad r=\kappa w,\qquad t=\kappa u,\qquad s=\kappa uw,
\]
reduces the quartic quadruple problem to the single equation \eqref{eq:special-quartic}. After homogenization this is Euler's surface.

\begin{proposition}\label{prop:euler-specialization}
Let $(X:Y:Z:W)\in\Pj^3(\Q)$ with $ZW\neq 0$, and define
\begin{equation}\label{eq:k-u-w-from-euler}
\kappa=\frac{W}{Z},\qquad u=\frac{X}{W},\qquad w=\frac{Y}{W}.
\end{equation}
Then $(X:Y:Z:W)$ lies on the Euler surface
\[
\Eul:\qquad X^4+Y^4=Z^4+W^4
\]
if and only if $k,u,w$ satisfy \eqref{eq:special-quartic}. Equivalently, if we set
\[
r=\frac{Y}{Z},\qquad s=\frac{XY}{ZW},\qquad t=\frac{X}{Z},\qquad v=0,
\]
then
\[
r=\kappa w,\qquad s=\kappa uw,\qquad t=\kappa u,
\]
and the compatibility conditions hold.
\end{proposition}

\begin{proof}
From \eqref{eq:k-u-w-from-euler} we have
\[
X=uW,\qquad Y=wW,\qquad Z=\frac{W}{\kappa}.
\]
Therefore
\[
X^4+Y^4=Z^4+W^4
\iff
u^4W^4+w^4W^4=\frac{W^4}{\kappa^4}+W^4,
\]
and, since $W\neq 0$, this is equivalent to
\[
\kappa^4(u^4+w^4)=1+\kappa^4,
\]
which is exactly \eqref{eq:special-quartic}. The remaining assertions are immediate from the definitions.
\end{proof}

We can now prove Proposition~\ref{prop:euler-to-quadruple} from the introduction.

\begin{proof}[Proof of Proposition~\ref{prop:euler-to-quadruple}]
Let $(X:Y:Z:W)\in\Eul(\Q)$ with $ZW\neq 0$, and define $\kappa,u,w,r,s,t,v$ as in Proposition~\ref{prop:euler-specialization}. Then $v=0$, $r=\kappa w$, $s=\kappa uw$, $t=\kappa u$, and \eqref{eq:special-quartic} holds. By Proposition~\ref{prop:auto-square-special}, the square condition in \eqref{eq:abcd-from-rstuvw} is automatic, and therefore the formulas \eqref{eq:abcd-from-rstuvw} produce a quartic quadruple.

Substituting
\[
r=\frac{Y}{Z},\qquad s=\frac{XY}{ZW},\qquad u=\frac{X}{W},\qquad v=0
\]
into \eqref{eq:abcd-from-rstuvw}, and choosing the sign of $a$ so that
\[
a=\frac{X^4-W^4}{Z^2W^2},
\]
we obtain exactly the formulas \eqref{eq:euler-map}. The displayed identities for $ab+1$, $ac+1$, $ad+1$, $bc+1$, $bd+1$, and $cd+1$ are then immediate.

It remains to determine when $a,b,c,d$ are pairwise distinct and nonzero. Since $ZW\neq 0$, the numbers $b$ and $d$ are automatically nonzero. Also,
\[
a=0 \iff X^4=W^4 \iff X^2=W^2,
\qquad
c=0 \iff Y^4=W^4 \iff Y^2=W^2.
\]
Further,
\[
a-b=\frac{X^4}{Z^2W^2},
\qquad
c-d=-\frac{X^4}{Z^2W^2},
\]
so $a=b$ and $c=d$ are both equivalent to $X=0$. Likewise,
\[
a-d=-\frac{Y^4}{Z^2W^2},
\qquad
c-b=\frac{Y^4}{Z^2W^2},
\]
using $X^4+Y^4=Z^4+W^4$, so $a=d$ and $c=b$ are both equivalent to $Y=0$. Finally,
\[
a-c=\frac{X^4-Y^4}{Z^2W^2},
\]
so $a=c$ is equivalent to $X^2=Y^2$. Since $bd+1=0$, we also have $bd=-1$, and therefore $b\neq d$. This proves that $a,b,c,d$ are pairwise distinct and nonzero if and only if
\[
XYZW(X^2-Y^2)(X^2-W^2)(Y^2-W^2)\neq 0,
\]
which is exactly the condition stated in Proposition~\ref{prop:euler-to-quadruple}.
\end{proof}

The Euler surface $\Eul$ is itself a classical object in algebraic geometry and arithmetic. As a smooth quartic surface in $\Pj^3$, it is a K3 surface. Over $\Q(\zeta_8)$ it becomes projectively equivalent to the Fermat quartic
\[
X^4+Y^4+Z^4+W^4=0,
\]
so one may import a great deal of geometry from that model. In particular, over $\overline{\Q}$ it has Picard number $20$, hence is a singular K3 surface, and it contains $48$ lines; see, for example, \cite{ShiodaFermat,SchuettShiodaVanLuijk}.

For the purposes of the present paper, the most relevant feature of \(\Eul\) is the existence of rational curves on the surface. Euler's classical parametrization already produces one such curve. Moreover, Swinnerton-Dyer stated a theorem describing the curves of arithmetic genus \(0\) on \(\Eul\), and Bremner has recently made this description explicit and effective; in particular, he determined all nontrivial parametrizations of degree below any prescribed bound, showing that there are \(86\) distinct nontrivial parametrizations of degree \(<50\) \cite{SwinnertonDyerAGNT,Bremner}. This indicates that the Euler surface provides a rich source of families of quartic Diophantine quadruples.
\section{Higher exponents and the Fermat--Euler surface}\label{sec:higher-exponents}

The argument of Section~\ref{sec:euler-surface} is not specific to the quartic case. 
If one replaces the exponent $4$ everywhere by an arbitrary integer $k\ge 2$, then the same construction leads naturally to the degree-$k$ Fermat--Euler surface
\[
\mathcal E_k:\qquad X^k+Y^k=Z^k+W^k.
\]
The only essential difference between the even and odd cases is that in odd degree one must impose an additional square condition on the ratio $W/Z$.

Let $k\ge 2$. By analogy with \eqref{eq:rstuvw-def}, we write
\begin{equation}\label{eq:rstuvw-def-general}
\begin{alignedat}{2}
ab+1&=r^k,  &\qquad ad+1&=u^k, \\
ac+1&=s^k,  &\qquad bd+1&=v^k, \\
bc+1&=t^k,  &\qquad cd+1&=w^k.
\end{alignedat}
\end{equation}
Then the same calculation as before gives
\begin{equation}\label{eq:abcd-from-rstuvw-general}
a^2=\frac{(r^k-1)(u^k-1)}{v^k-1},\qquad
b=\frac{r^k-1}{a},\qquad
c=\frac{s^k-1}{a},\qquad
d=\frac{u^k-1}{a},
\end{equation}
together with the compatibility relations
\begin{equation}\label{eq:compatibility-second-general}
(r^k-1)(w^k-1)=(t^k-1)(u^k-1),
\end{equation}
\begin{equation}\label{eq:compatibility-third-general}
(r^k-1)(w^k-1)=(s^k-1)(v^k-1),
\end{equation}
and the square condition
\begin{equation}\label{eq:square-general}
(r^k-1)(u^k-1)(v^k-1)\in (\Q^\times)^2.
\end{equation}

As in the quartic case, we now impose
\[
v=0,\qquad ru=tw.
\]
Writing
\[
\kappa=\frac{r}{w}=\frac{t}{u},
\]
we have
\[
r=\kappa w,\qquad t=\kappa u.
\]
If we further set
\[
s=\kappa uw,
\]
then the same computation as in Proposition~\ref{prop:quartic-reduction} shows that, away from the diagonal locus $u^k=w^k$, the two compatibility conditions \eqref{eq:compatibility-second-general} and \eqref{eq:compatibility-third-general} are both equivalent to
\begin{equation}\label{eq:general-special-ferm}
u^k+w^k=1+\frac{1}{\kappa^k}.
\end{equation}
Moreover, the square condition \eqref{eq:square-general} becomes
\begin{equation}\label{eq:general-square-special}
(r^k-1)(u^k-1)(v^k-1)=\frac{(r^k-1)^2}{\kappa^k}.
\end{equation}
Thus, if $k$ is even, the square condition is automatic, while if $k$ is odd it is automatic provided $\kappa$ is a square in $\Q$.

Equation \eqref{eq:general-special-ferm} is the affine form of the Fermat--Euler surface
\[
\mathcal E_k:\qquad X^k+Y^k=Z^k+W^k.
\]
Indeed, if
\[
\kappa=\frac{W}{Z},\qquad u=\frac{X}{W},\qquad w=\frac{Y}{W},
\]
then \eqref{eq:general-special-ferm} is equivalent to
\[
X^k+Y^k=Z^k+W^k.
\]

We summarize the outcome in the following theorem.

\begin{theorem}\label{thm:general-euler-to-quadruple}
Let $k\ge 2$ and let
\[
\mathcal E_k:\qquad X^k+Y^k=Z^k+W^k.
\]

\smallskip

\noindent
\textup{(i) The even case.}
Assume that $k=2m$ is even, and let $(X:Y:Z:W)\in \mathcal E_k(\Q)$ with $ZW\neq 0$. Define
\begin{equation}\label{eq:general-euler-map-even}
a=\frac{X^k-W^k}{Z^mW^m},\qquad
b=-\Bigl(\frac{W}{Z}\Bigr)^m,\qquad
c=\frac{Y^k-W^k}{Z^mW^m},\qquad
d=\Bigl(\frac{Z}{W}\Bigr)^m.
\end{equation}
Then
\[
ab+1=\Bigl(\frac{Y}{Z}\Bigr)^k,\qquad
ac+1=\Bigl(\frac{XY}{ZW}\Bigr)^k,\qquad
ad+1=\Bigl(\frac{X}{W}\Bigr)^k,
\]
\[
bc+1=\Bigl(\frac{X}{Z}\Bigr)^k,\qquad
bd+1=0,\qquad
cd+1=\Bigl(\frac{Y}{W}\Bigr)^k.
\]

\smallskip

\noindent
\textup{(ii) The odd case.}
Assume that $k$ is odd, and let $(X:Y:Z:W)\in \mathcal E_k(\Q)$ with $ZW\neq 0$. Suppose in addition that
\[
\frac{W}{Z}=\lambda^2
\]
for some $\lambda\in \Q^\times$. Define
\begin{equation}\label{eq:general-euler-map-odd}
a=\frac{X^k-W^k}{\lambda^k Z^k},\qquad
b=-\lambda^k,\qquad
c=\frac{Y^k-W^k}{\lambda^k Z^k},\qquad
d=\lambda^{-k}.
\end{equation}
Then
\[
ab+1=\Bigl(\frac{Y}{Z}\Bigr)^k,\qquad
ac+1=\Bigl(\frac{XY}{ZW}\Bigr)^k,\qquad
ad+1=\Bigl(\frac{X}{W}\Bigr)^k,
\]
\[
bc+1=\Bigl(\frac{X}{Z}\Bigr)^k,\qquad
bd+1=0,\qquad
cd+1=\Bigl(\frac{Y}{W}\Bigr)^k.
\]

\smallskip

In both cases, if
\begin{equation}\label{eq:general-nondeg}
XYZW\,(X^k-Y^k)(X^k-W^k)(Y^k-W^k)\neq 0,
\end{equation}
then the resulting numbers $a,b,c,d$ are pairwise distinct and nonzero, and hence form a $k$th power rational Diophantine quadruple.
\end{theorem}

\begin{proof}
The proof is formally identical to the proof of Proposition~\ref{prop:euler-to-quadruple}. 
One starts from \eqref{eq:rstuvw-def-general}--\eqref{eq:square-general}, imposes
\[
v=0,\qquad r=\kappa w,\qquad t=\kappa u,\qquad s=\kappa uw,
\]
and uses \eqref{eq:general-special-ferm} together with \eqref{eq:general-square-special}. 
In the even case the square condition is automatic; in the odd case it is automatic once $\kappa=W/Z$ is a square in $\Q$, say $\kappa=\lambda^2$. 
Substituting
\[
r=\frac{Y}{Z},\qquad s=\frac{XY}{ZW},\qquad t=\frac{X}{Z},\qquad u=\frac{X}{W},\qquad v=0,\qquad w=\frac{Y}{W}
\]
into \eqref{eq:abcd-from-rstuvw-general} yields the formulas
\eqref{eq:general-euler-map-even} and \eqref{eq:general-euler-map-odd}. 
The verification of the displayed identities for $ab+1,\dots,cd+1$ and of the nondegeneracy criterion \eqref{eq:general-nondeg} is the same as in the quartic case, with $4$ replaced by $k$.
\end{proof}

\begin{remark}
For $k=3$, Theorem~\ref{thm:general-euler-to-quadruple} recovers the cubic family of Batta--Szikszai--Tengely \cite{BST}. Indeed, for $K\in\Q^\times\setminus\{\pm1\}$ consider the rational point
\[
(X:Y:Z:W)=\bigl(K^2(K^6+2):-(2K^6+1):K^6-1:K^2(K^6-1)\bigr)
\]
on
\[
\mathcal E_3:\qquad X^3+Y^3=Z^3+W^3.
\]
Since
\[
\frac{W}{Z}=K^2
\]
is a square in $\Q$, the odd case of Theorem~\ref{thm:general-euler-to-quadruple} applies with $\lambda=K$ and yields
\[
a=\frac{9K^3(K^{12}+K^6+1)}{(K^6-1)^3},\qquad
b=-K^3,
\]
\[
c=-\frac{K^{24}+5K^{18}+15K^{12}+5K^6+1}{K^3(K^6-1)^3},\qquad
d=\frac1{K^3}.
\]
If we now set $t=K^{-3}$, then this becomes
\[
\left\{
-\frac{9(t^5+t^3+t)}{t^6-3t^4+3t^2-1},
-\frac1t,
\frac{t^8+5t^6+15t^4+5t^2+1}{t^7-3t^5+3t^3-t},
t
\right\},
\]
which, after reordering the elements, is precisely the cubic quadruple family from \cite[Theorem~4.11]{BST}.
\end{remark}

For exponents $k>4$, the arithmetic of the Fermat--Euler surface
\[
\mathcal E_k:\qquad X^k+Y^k=Z^k+W^k
\]
appears to be substantially more rigid than in the quartic case. Indeed, $\mathcal E_k$ is a smooth surface of degree $k$ in $\Pj^3$, hence of general type for $k\geq 5$. In particular, the Bombieri--Lang conjecture predicts that the set of rational points on $\mathcal E_k$ outside the union of its curves of genus $0$ and $1$ should not be Zariski dense; in the form relevant to equal sums of powers, one expects only finitely many rational points away from such low-genus curves. Thus, for $k>4$, the existence of infinite families of $k$th power rational Diophantine quadruples coming from Theorem~\ref{thm:general-euler-to-quadruple} should be viewed as a problem about finding special rational or elliptic curves on $\mathcal E_k$. In the case when $k=6$, Newton and Rouse remark that it seems likely that no positive integer can be written as a sum of two rational sixth powers in more than one way, and they cite Ekl's computer search, which found no nontrivial integer solutions to
\[
a^6+b^6=c^6+d^6
\]
with $a\neq c,d$ and $a^6+b^6<7.25\times 10^{24}$ \cite{Ekl,NewtonRouse}. This provides some computational evidence that nontrivial rational points on $\mathcal E_6$ may be rather scarce.

\section{Acknowledgements}

A.A. and M.K. were supported by the Croatian Science Foundation under the project no.\ IP-2022-10-5008 (TEBAG). M.K. acknowledges support from the project “Implementation of cutting-edge research and its application as part of the Scientific Center of Excellence for Quantum and Complex Systems, and Representations of Lie Algebras”, Grant No.\ PK.1.1.10.0004, co-financed by the European Union through the European Regional Development Fund -- Competitiveness and Cohesion Programme 2021-2027. This research was funded by the European Union NextGenerationEU through the National Recovery and Resilience Plan 2021-2026. Institutional grants of University of Zagreb Faculty of Science (IK IA 1.1.3. Impact4Math) and University of Zagreb Faculty of Electrical Engineering and Computing, Value-aligned and interpretable optimization and reasoning (VALOR). Institutional projects “BELLGI – Bellman functions, graphs and computability” (100‑038/26) and “TopoPha – Topological Phases” (100‑021/26) of the University of Zagreb Faculty of Civil Engineering.
This research was carried out using the advanced computing service provided by the University of Zagreb University Computing Centre - SRCE.


\begin{thebibliography}{99}


\bibitem{Andrasek}
A.~Andra\v{s}ek,
\newblock \emph{On Regular Higher Power Rational Diophantine Triples},
\newblock preprint arxiv:2604.17018.


\bibitem{BST}
G.~Batta, M.~Szikszai and S.~Tengely,
\newblock Higher power rational Diophantine tuples,
\newblock \emph{Ramanujan J.} \textbf{68} (2025), Art.~63.


\bibitem{Bremner}
A.~Bremner,
\newblock On the quartic surface $x^4+y^4=z^4+w^4$,
\newblock \emph{Acta Arith.} \textbf{220} (2025), no.~3, 209--248.

\bibitem{BD}
Y.~Bugeaud and A.~Dujella,
On a problem of Diophantus for higher powers,
\emph{Mathematical Proceedings of the Cambridge Philosophical Society}
\textbf{135} (2003), no.~1, 1--10.

\bibitem{BF}
D.~Byeon and C.~Fuchs,
\newblock \emph{Cube Diophantine triples and elliptic curves},
\newblock preprint arXiv:2510.01896, to appear in \emph{Acta Arithmetica}.



\bibitem{DujellaBook}
A.~Dujella,
\newblock \emph{Diophantine $m$-tuples and elliptic curves},
\newblock Springer, Cham, 2024.

\bibitem{DK}
A.~Dujella and M.~Kazalicki,
\newblock More on Diophantine sextuples,
\newblock in \emph{Number Theory --- Diophantine Problems, Uniform Distribution and Applications},
Festschrift in honour of Robert F.~Tichy's 60th birthday,
Springer, 2017, pp.~227--235.

\bibitem{DKMS}
A.~Dujella, M.~Kazalicki, M.~Miki\'c and M.~Szikszai,
\newblock There are infinitely many rational Diophantine sextuples,
\newblock \emph{Int. Math. Res. Not. IMRN} 2017(2) (2017), 490--508.

\bibitem{DKP-sq}
A.~Dujella, M.~Kazalicki and V.~Petri\v{c}evi\'c,
\newblock There are infinitely many rational Diophantine sextuples with square denominators,
\newblock \emph{J. Number Theory} \textbf{205} (2019), 340--346.

\bibitem{DKP-strong}
A.~Dujella, M.~Kazalicki and V.~Petri\v{c}evi\'c,
\newblock Rational Diophantine sextuples with strong pair,
\newblock \emph{Rev. R. Acad. Cienc. Exactas F\'{\i}s. Nat. Ser. A Mat. RACSAM} \textbf{119} (2025), Article~36.

\bibitem{Ekl}
R.~L.~Ekl,
\newblock \emph{New results in equal sums of like powers},
\newblock \emph{Math. Comp.} \textbf{67} (1998), 1309--1315.

\bibitem{NewtonRouse}
A.~Newton and J.~Rouse,
\newblock \emph{Integers that are sums of two rational sixth powers},
\newblock \emph{Canad. Math. Bull.} \textbf{66} (2023), no.~1, 166--177.


\bibitem{SchuettShiodaVanLuijk}
M.~Sch\"utt, T.~Shioda and R.~van Luijk,
\newblock Lines on Fermat surfaces,
\newblock \emph{J. Number Theory} \textbf{130} (2010), no.~9, 1939--1963.


\bibitem{ShiodaFermat}
T.~Shioda,
\newblock On the Picard number of a Fermat surface,
\newblock \emph{J. Fac. Sci. Univ. Tokyo Sect. IA Math.} \textbf{28} (1982), no.~3, 725--734.

\bibitem{SwinnertonDyerAGNT}
H.~P.~F. Swinnerton-Dyer,
\newblock \emph{Applications of algebraic geometry to number theory},
\newblock in \emph{Number Theory Institute}
(Proc. Sympos. Pure Math., Vol.~XX, State Univ. of New York, Stony Brook, N.Y., 1969),
pp.~1--52, Amer. Math. Soc., Providence, RI, 1971.

\end{thebibliography}
\end{document}